\documentclass[11pt, a4paper,DIV=12]{scrartcl}

\usepackage{amsthm,amsfonts,amsmath,latexsym,mathrsfs,amssymb}
\usepackage{mathtools}

\usepackage[T1]{fontenc}
\usepackage{newtxtext}
\usepackage{newtxmath}
\usepackage{xcolor}
\usepackage{thmtools}

\usepackage{microtype}

\usepackage{hyperref}
\hypersetup{colorlinks=true, linkcolor=blue!60!black, citecolor=blue!60!black, urlcolor=blue!60!black}

\newtheorem{theorem}{Theorem}[section]
\newtheorem{lemma}[theorem]{Lemma}
\newtheorem{result}[theorem]{Result}
\newtheorem{proposition}[theorem]{Proposition}
\newtheorem{corollary}[theorem]{Corollary}

\theoremstyle{definition}
\newtheorem{definition}[theorem]{Definition}

\theoremstyle{remark}
\newtheorem{remark}[theorem]{Remark}

\numberwithin{equation}{section}

\newcommand{\ol}{\overline}
\newcommand{\eps}{\varepsilon}

\newcommand{\id}{\operatorname{id}}
\newcommand{\Ind}{\operatorname{index}}
\newcommand{\Aut}{\operatorname{\textsf{Aut}}}

\newcommand{\wind}{\operatorname{wind}}

\newcommand{\bdy}{\partial}

\newcommand{\disk}{\mathbb{D}}

\newcommand{\T}{\bdy \disk}

\newcommand{\C}{\mathbb{C}}
\newcommand{\R}{\mathbb{R}}

\newcommand{\D}{\mathbb{D}}

\usepackage{enumitem}

\setlist[enumerate]{
  itemsep=2pt,
  topsep=2pt,
  parsep=0pt,
  partopsep=0pt,
}

\makeatletter
\def\blfootnote{\gdef\@thefnmark{}\@footnotetext}
\makeatother

\begin{document}

\title{Holomorphic mappings of finitely connected planar domains}

\author{Jaikrishnan Janardhanan \thanks{The author is supported by a
grant from ANRF: \textbf{ANRF/ARGM/2025/000619/MTR} \\
\href{mailto:jaikrishnan@iitpkd.ac.in}{jaikrishnan@iitpkd.ac.in}\\
Department of Mathematics, Indian Institute of Technology Palakkad,
Palakkad, Kerala-678623, India} }

\date{\today}

\maketitle

\begin{abstract}
  Results on holomorphic mappings of finitely connected planar domains are often
  proved using Koebe's famous circle mapping theorem. A prominent example is
  Julia's bound on the size of the automorphism group of finitely connected
  planar domains of connectivity $k\geq3$ and the sharp bound obtained later by
  Heins. The purpose of this article is to explore whether it is possible to
  give a proof of these results, and several other related results, without
  using Koebe's theorem. The main result is a rigidity
  theorem for proper holomorphic maps between finitely connected domains of
  connectivity higher than $2$ that immediately yields a bound on the number of
  proper holomorphic mappings between two such domains in terms of their
  connectivities. Using the same ideas, we also provide a new elementary proof
  of the planar Riemann--Hurwitz formula that does not use the Euler
  characteristic. Most of our proofs rely only on the Riemann mapping theorem,
  the Schwarz reflection principle, and the argument principle.
\end{abstract}

\blfootnote{\textup{2020} \textit{Mathematics Subject Classification}:
Primary: 30C35, 32H35; Secondary: 30C20, 30C85.}


\section{Introduction}
The primary purpose of this article is to present an elementary
treatment of some classical and some newer results about conformal and
proper holomorphic mappings of finitely connected planar domains. Our
starting point is our earlier article \cite{bharathi} in which we showed
that any conformal self-map of a planar domain that has $3$ fixed points
must be the identity map. We also gave ``simple'' proofs of several
other classical results about conformal maps. These proofs relied on
Koebe's famous circle mapping theorem that asserts that any bounded
finitely connected domain is biholomorphic to the unit disk with
finitely many points and finitely many pairwise disjoint closed subdisks
removed. We will show in this article that one does \textbf{not} require
Koebe's theorem in the proofs of the results in \cite{bharathi}; see
Remark~\ref{R:old} below. In fact, we will show that the proofs require
only standard facts from complex analysis, namely, the Riemann mapping
theorem, the Schwarz reflection principle, and the argument principle.
It turns out that the same circle of ideas also yields new proofs of
several classical results about conformal mappings of finitely connected
planar domains as well as new results about proper holomorphic mappings
of such domains. 

To explain our results, we need to introduce the following terminology
and notation that we will use throughout the article. A bounded
finitely connected planar domain is called \emph{admissible} if every
nondegenerate boundary component is a smooth real-analytic Jordan
curve. An admissible domain is called \emph{unpunctured} if it has no
degenerate boundary components. Whenever \(G\) and \(D\) are admissible
domains occurring as the source and target of a map, respectively, we
write
\[
  \partial G=\Sigma_0\sqcup\cdots\sqcup\Sigma_{m-1},
  \qquad
  \partial D=\Gamma_0\sqcup\cdots\sqcup\Gamma_{k-1},
\]
where \(\Sigma_0\) and \(\Gamma_0\) are the outer boundary components.
Thus \(m\) and \(k\) denote the connectivities of \(G\) and \(D\),
respectively. We will always use the arc-length
parametrisation for the boundary Jordan curves. It is clear that this
parametrisation is both real analytic and regular. We will abuse
notation and use the same symbol to denote a boundary curve and its
arc-length parametrisation.

Another motivation for this article comes from the study of the
size of the automorphism group of finitely connected domains. A precise
upper bound once again follows immediately from the circle mapping
theorem. This seems to be due to Julia \cite[pp. 68-69]{Julia}.

\begin{result}\label{R:julia}
  Let $D$ be a finitely connected domain of connectivity $k \geq 3$.
  Then
  \[
    |\Aut(D)| \leq k(k-1)(k-2).
  \]
\end{result}

Heins \cite{Heins2} later obtained the best upper bound for the size of
the automorphism group of a finitely connected planar domain in terms
of the connectivity.

\begin{result}[Heins, Theorem~B of \cite{Heins2}]
\label{R:heins}
Let \(D\subset\mathbb C\) be a finitely connected domain of connectivity
\(k \geq 3\). Then $|\Aut(D)| \leq 2k$ except when $k = 4,6,8,12,$ or $20$. In
these cases we have
\[
 |\Aut(D)| \leq
 \begin{cases}
  12,&k=4,\\
  24,&k=6\ \text{or}\ 8,\\
  60,&k=12\ \text{or}\ 20.
 \end{cases}
\]
Each one of these bounds is attained.
\end{result}

Again the key tool used in the proof is Koebe's circle mapping theorem.
At the end of his article, Heins remarks that it would be interesting to
deduce his result without using Koebe's theorem. In an earlier article
\cite{Heins1}, Heins had already outlined a possible approach using
harmonic measure to show the weaker result that the automorphism group
of such domains is finite.

In this article, we give simple proofs of Julia's and Heins's results.
Notably, we do not use Koebe's circle mapping theorem. We now briefly
describe the ideas behind our proofs. Julia's result is straightforward
if every component of the complement of $D$ is a point. In all other cases, by a
standard argument using only the Riemann mapping theorem, we may assume
that the domain $D$ is admissible. One can easily remove the punctures.
The Schwarz reflection principle now shows that the automorphism extends
holomorphically to a larger domain. We now show that the extension of a
nontrivial automorphism cannot have fixed points on the boundary; see
Proposition~\ref{prop:no-boundary-fixed-point} below. We give two proofs
of this proposition. The first uses the compactness of the automorphism
group that follows immediately from the famous Wong--Rosay theorem and
Cartan's uniqueness theorem. The second argument uses the aforementioned
idea of Heins \cite{Heins1}. Once boundary fixed points are excluded,
the argument principle allows us to count the number of interior fixed
points, which turns out to be $2 - b(f)$, where $b(f)$ is the
number of boundary components fixed setwise by the automorphism $f$.
This count yields the required bound. This proof is adapted from the
arguments in the seminal work of He and Schramm; see
\cite[Section~2]{Schramm}. The fixed-point count also allows us to give
a very simple proof of Result~\ref{R:heins}. It is clear that any
automorphism of a finitely connected domain induces a homeomorphism of
the sphere after we identify each boundary component with a point. Thus
the group of automorphisms is a compact group that acts on the sphere. A
result of Kerékjártó \cite{Kerékjártó} now shows that this group must be
conjugate to a subgroup of the special orthogonal group. As the group
under consideration is already finite, the classification of finite
subgroups of the special orthogonal group applies and this finishes the
proof. In our proof, instead of appealing to Kerékjártó's theorem, we 
adapt the classical pole counting argument Klein used to
classify finite subgroups of the special orthogonal group. In our
context, the $2$ fixed points act as the poles.

The same circle of ideas can be used to study
proper holomorphic maps of finitely connected domains. Let \(G\) and
\(D\) be finitely connected planar domains of connectivities \(m\) and
\(k\), respectively, and let $f\colon G\longrightarrow D$ be a proper
holomorphic map of degree \(d\). If \(r\) denotes the number of critical
points of \(f\), counted with multiplicity, then the planar
Riemann--Hurwitz formula says
\[
    m-2=d(k-2)+r.
\]
Steinmetz \cite{Steinmetz} gave an elementary proof of this formula
which avoids the Euler characteristic and proceeds by cutting the
domains along cross-cuts. We give a different proof using the argument
principle. As an immediate consequence, every proper holomorphic
self-map of a planar domain of finite connectivity at least three is an
automorphism. The latter result was shown by Mueller and Rudin
\cite{Rudin}. Their proof uses the fact that the domain in question can
be mapped biholomorphically to a circular-slit domain. While this is
easier to prove than Koebe's circle mapping theorem, it is nevertheless
a highly nontrivial result. Our proof of the Riemann--Hurwitz theorem
requires only the Riemann mapping theorem and the argument principle.

We now come to the main result of this article. We consider the
collection of all proper holomorphic maps between two fixed finitely
connected planar domains. In \cite{Steinmetz}, it is remarked that using
circular-slit uniformization, it is easy to see that this collection is
finite whenever the target has connectivity at least three. A complete
proof along these lines was subsequently given by Hemasundar
\cite{Hemasundar}. A precise bound had already been established by
Jenkins and Suita \cite{Jenkins}, who showed that the number of proper
holomorphic maps from an $m$-connected planar domain, $m \geq 2$, onto a
fixed planar domain of connectivity greater than two is at most
\[
  (m - 2)2^{(4m - 6)},
\]
using circular-slit mappings and extremal metrics, together with the
Riemann--Hurwitz formula. We give a much better bound using harmonic
measures. 
\begin{restatable}{theorem}{Proper}
\label{thm:proper-count}
Let \(G\) and \(D\) be unpunctured admissible domains, and suppose that
\(k\ge3\).
For fixed disjoint collections \(A_1,A_2\) of boundary components of \(G\),
there are at most \(k-2\) proper holomorphic maps \(f:G\to D\) for which
exactly the curves in \(A_j\) are mapped onto \(\Gamma_j\), \(j=1,2\).
Consequently,
\[
 \#\{f:G\to D:f\text{ is proper holomorphic}\} \le\ (k-2)\bigl(3^m-3\cdot2^m+3\bigr).
\]
If $G = D$ then every proper holomorphic map is an automorphism and we
have the bound
\[
  |\Aut(D)| \leq k(k-1)(k-2).
\]
\end{restatable}
The connectivity of the target being at least three is crucial in our proof as
this allows us to produce a pair of ``linearly independent'' harmonic measures
on the target domain. The above theorem subsumes Julia's theorem
and the aforementioned result by Mueller and Rudin. Similar arguments
can also be used to prove the Riemann--Hurwitz theorem. Therefore, in a strong
sense, Theorem~\ref{thm:proper-count} can be viewed as a unification of the
other results we prove in this article.

The paper is organised as follows. Section~2 collects the basic facts about
proper holomorphic maps needed later, including boundary correspondence and
holomorphic extension across analytic boundary components. Section~3
establishes the fixed-point count for automorphisms and uses it to prove the
bounds of Julia and Heins. Section~4 gives an elementary proof of the planar
Riemann--Hurwitz formula. Finally, Section~5 uses harmonic measure to bound the
number of proper holomorphic maps between two fixed admissible domains. The
polarising adjective \emph{elementary} occurs several times in this article. We
use the word only when what follows is a consequence of the facts usually
presented in basic courses in complex
analysis, real analysis and topology.

\section{Basic facts about proper holomorphic maps of planar domains}

The purpose of this section is to recall well-known results about proper
holomorphic mappings. Recall that a proper map between Hausdorff
topological spaces is a continuous map such that the preimage of every
compact set is compact. For domains in Euclidean spaces, this is easily
seen to be equivalent to the following: a continuous map $f:D\to G$,
where $D\subset\R^n$ and $G\subset\R^m$, is proper if, for any sequence
$x_n\in D$ with no limit points in $D$, the image sequence $f(x_n)$ has
no limit points in $G$. The results in this section are standard, but we
will state them in the form we need. We give proofs whenever we
could not find a source that proves the results in our specific
formulation. We begin with the standard fact that proper holomorphic
mappings between equidimensional complex spaces are finite branched
coverings. An elementary proof in dimension $1$ can be found in
\cite[Theorem~1.1]{Rudin}.

\begin{result}
  Let $G$ and $D$ be domains in $\C$ and let $f:G \to D$ be a
  proper holomorphic map. Then there is an integer $d \ge 1$, called the
  \emph{degree of $f$}, such that for each $w \in D$, the preimage
  $f^{-1}(w)$ comprises exactly $d$ points, counting multiplicities. In
  particular, $f$ is surjective.
\end{result}

As this paper deals mostly with finitely connected planar domains, we
will make use of the following simple consequence of the Riemann mapping
theorem that allows us to restrict attention to bounded domains in $\C$
that are admissible. A proof can be found in
\cite[Theorem~9.29]{Zakeri}.

\begin{result}
  Let $D$ be a bounded finitely connected domain in $\C$ of connectivity
  $k$, i.e., the complement of $D$ in the Riemann sphere has exactly $k$
  connected components. Then $D$ is biholomorphic to an admissible
  subdomain of the unit disc $\D$ whose outer boundary component is $\T$.
\end{result}

Henceforth, we will represent all finitely connected planar domains by
admissible domains as in the above result. As one anticipates, the
punctures can be removed without much difficulty. To get to this, we
need to first study what happens to the boundary components under a
proper holomorphic map. As we have assumed that the nondegenerate
boundary components are all real-analytic Jordan curves, our task is
substantially simplified thanks to the Schwarz reflection principle. We
will rely on the following version from Lang's elegant treatment of the
reflection principle in \cite[Chapter~IX]{Lang}. The merit of this
version is that it does not require the map in question to extend
continuously to the closure, and a continuous extension is a byproduct
of the proof.

\begin{result}[a paraphrase of Theorem~2.5 in Chapter~IX of \cite{Lang}]\label{R:lang}
  Let $G$ be a simply connected unpunctured admissible domain and let
  $f:G \to \disk$ be the Riemann map. Then $f$ extends as a
  biholomorphic map from a domain
  $\widetilde{G}$ that contains $\ol{G}$ onto the disc
  $\{z\in\C:|z|<R\}$ for some $R>1$.
\end{result}

The above result immediately shows that we can find a nice analytic
collar around a smooth real-analytic Jordan curve.
\begin{corollary}
 Let $\Gamma\subset\C$ be a smooth real-analytic Jordan curve. We can
 find numbers $R,r$ with \(0<r<1<R\) and an injective holomorphic map
\[
      \chi:\{z \in \C: r<|z|<R\}\longrightarrow\C
\]
such that \(\chi(\bdy \disk)=\Gamma\) and $\chi\left(\{z: r < |z| < 1\}\right)$
is contained in the bounded component determined by $\Gamma$.
\end{corollary}

\begin{remark}
 Let $D$ be an unpunctured admissible domain.
 Each boundary component gives rise to a biholomorphism from an annulus.
 By shrinking the annuli, we can also assume that the images of the
 annuli are pairwise disjoint. Furthermore, by precomposing with the
 appropriate inversion, we can ensure that each map $\chi$ sends the portion
 of its annulus inside the unit disc into the domain $D$. Having made
 these adjustments, we will refer to the images of $\{z: r < |z| < 1\}$
 under $\chi$ as \emph{one-sided analytic collars}.
\end{remark}

We now analyse what happens to the boundary components under a proper holomorphic
map. The following lemma is also recorded in
\cite[Section~3]{Rudin} in a different formulation.

\begin{lemma}\label{L:bdycorres} Let $G$ and $D$ be admissible domains
  and let $f:G \to D$ be a proper holomorphic map. Then, for each
  $i\in\{0,\ldots,m-1\}$, we can find a fixed $j\in\{0,\ldots,k-1\}$
  such that as $z\to\Sigma_i$ we have $f(z)\to\Gamma_j$. In particular,
  $k\leq m$.
\end{lemma}

\begin{proof}
  For a nondegenerate boundary component $\Gamma_j$, let $\mathcal{V}_j$
  denote the one-sided analytic collar around $\Gamma_j$ in $D$ obtained
  from the previous corollary. In the case that $\Gamma_j$ is a degenerate
  boundary component, we can take $\mathcal{V}_j$ to be a small
  punctured disc around the point $\Gamma_j$. By shrinking, we may assume
  that the collars $\mathcal{V}_j$ are pairwise disjoint. Let $H :=
  D \setminus \bigcup_{j=0}^{k-1} \mathcal{V}_j$. Then $H$ is a compact
  subset of $D$. Since $f$ is proper, the preimage $f^{-1}(H)$ is a
  compact subset of $G$. Let $\mathcal{U}_i$ be a small enough one-sided
  analytic collar (or punctured disc) around $\Sigma_i$ in $G$ that is
  disjoint from $f^{-1}(H)$. Then $f(\mathcal{U}_i)$ is a connected
  subset of $D$ that is disjoint from $H$. Hence, $f(\mathcal{U}_i)$ is
  contained in one of the collars $\mathcal{V}_j$. The rest of the lemma now
  follows. That $k\leq m$ is immediate from the fact that $f$ is
  surjective.
\end{proof}

\begin{remark}
  A proper holomorphic map $f:G\to D$ between two admissible domains
  induces a map $\sigma_f:\{0,\ldots,m-1\}\to\{0,\ldots,k-1\}$ that
  associates to each boundary component of $G$ a boundary component of
  $D$.
\end{remark}

We will now use the Schwarz reflection
principle to show that a proper holomorphic map between two unpunctured
admissible domains
extends holomorphically to a neighbourhood of the closure of the domain.

\begin{theorem}\label{thm:proper-extension} Let \(G\) and \(D\) be
unpunctured admissible domains. Every proper holomorphic map \(f:G\to D\)
extends holomorphically to a neighbourhood of \(\ol G\).
The extension maps \(\partial G\) into \(\partial D\), has no critical
point on \(\partial G\), and maps each boundary component of \(G\) onto
a boundary component of \(D\) as an orientation-preserving finite
covering.
\end{theorem}

\begin{proof}

Fix a boundary component \(\Sigma_i\) of \(G\), and let
\(\Gamma_{\sigma(i)}\) be the component associated with it by the previous
lemma. Choose one-sided analytic collars
\[
 \alpha:\{r<|\zeta|<1\}\to U_{\Sigma_i},
 \qquad
 \beta:\{s<|\eta|<1\}\to U_{\Gamma_{\sigma(i)}},
\]
around $\Sigma_i$ and $\Gamma_{\sigma(i)}$, respectively. By shrinking the
collar $U_{\Sigma_i}$ if necessary, we may assume that
\(f(U_{\Sigma_i})\subset U_{\Gamma_{\sigma(i)}}\). The map
\[
        H(\zeta)=\beta^{-1}\!\bigl(f(\alpha(\zeta))\bigr)
\]
is well-defined and holomorphic on the annulus $\{r<|\zeta|<1\}$.
Furthermore, as $|\zeta| \to 1$, we have $|H(\zeta)| \to 1$.
Thus the Schwarz reflection principle gives a holomorphic extension of
\(H\) to a neighbourhood of \(\T\) (see the proof of
Result~\ref{R:lang}). Composing with \(\beta\) and \(\alpha^{-1}\) gives
a holomorphic extension of \(f\) to a neighbourhood of \(\Sigma_i\) that maps
\(\Sigma_i\) into \(\Gamma_{\sigma(i)}\). Repeating this for each boundary
component of \(G\)
gives a holomorphic extension of \(f\) to a neighbourhood of \(\ol G\)
that maps \(\partial G\) into \(\partial D\). We continue to call this
extension \(f\).

Let $p \in \bdy G$ and $q := f(p)$. By applying a biholomorphic change
of coordinates near $p$ and $q$, we may assume that \(p=q=0\) and that
$f$ maps a small disc centred at $0$ into another small disc centred
at $0$ such that the upper semidisc is mapped into the upper half-plane
and the real axis is mapped into the real axis. Let $\ell$ be the
multiplicity of $f$ at $0$. We can write
\[
  f(z) = a_\ell z^\ell + O(z^{\ell+1}) \quad \text{as } z \to 0,
\]
where $a_\ell \neq 0$. Furthermore, as $f$ maps the real axis into
itself, we must have $a_\ell \in \R$. If $\ell > 1$, then it is clear
that for an appropriate choice of $z$ sufficiently close to $0$ in the
upper semidisc, we have $f(z)$ in the lower half-plane. This contradicts
the fact that $f$ maps the upper semidisc into the upper half-plane.
Therefore, each $f|_{\Sigma_i}$ is an open map and its image is a
connected subset of $\Gamma_{\sigma(i)}$. Hence $f|_{\Sigma_i}$ is
surjective onto $\Gamma_{\sigma(i)}$. The fact that $f|_{\Sigma_i}$ is
an orientation-preserving covering map onto $\Gamma_{\sigma(i)}$ is
elementary.

\end{proof}

By applying Theorem~\ref{thm:proper-extension} to both a biholomorphism
and its inverse, the following corollary is immediate.

\begin{corollary}
\label{cor:biholo-extension}
If \(f:G\to D\) is biholomorphic between unpunctured admissible domains,
then \(f\) extends biholomorphically from
a neighbourhood of \(\ol G\) onto a neighbourhood of \(\ol D\).
\end{corollary}

Now we can show that a proper holomorphic map between two finitely
connected planar domains extends to a proper holomorphic map between the
domains obtained by filling in the punctures of the two domains.

\begin{lemma}
  Let $G$ and $D$ be admissible domains and let
  $f:G \to D$ be a proper holomorphic map. Then $f$ extends to
  a proper holomorphic map between the domains obtained by filling in the
  punctures of $G$ and $D$, respectively.
\end{lemma}

\begin{proof}
  Riemann's removable singularity theorem implies that $f$ extends
  holomorphically to the punctures of $G$. Denote this extension by
  $\widetilde{f}$. Let $p$ be a puncture of $G$. By properness,
  $\widetilde{f}(p) \in \bdy D$. Furthermore, $\widetilde{f}(p)$ must be
  a puncture of $D$, for otherwise the open mapping theorem would be
  violated. Let $\widetilde{G}$ be the domain obtained by filling in the
  punctures of $G$ and let $\widetilde{D}$ be $\widetilde{f}(\widetilde
  G)$. We need to show that $\widetilde{D}$ has no punctures. Let $p$ be
  a puncture of $\widetilde{D}$. We can find a sequence $\{z_n\} \subset
  G$ such that $\widetilde{f}(z_n) \to p$. We may assume that $\{z_n\}$
  converges to a point on some nondegenerate boundary component of $G$,
  say $\Sigma_i$. Consequently, Lemma~\ref{L:bdycorres} shows that for any
  sequence $\{w_n\} \subset G$ such that $w_n \to \Sigma_i$, we have
  $\widetilde{f}(w_n) \to p$. We now apply the Schwarz reflection principle
  to the one-sided analytic collar around $\Sigma_i$ to see that
  $\widetilde{f} - p$ extends holomorphically to a neighbourhood of
  $\Sigma_i$. This contradicts the identity theorem, and hence $\widetilde{D}$
  has no punctures. We have, in fact, shown that the nondegenerate
  boundary curves of $G$ \emph{cannot} be mapped to the punctures of
  $D$. It is now evident that $\widetilde{f}$ is proper.
\end{proof}

For the rest of the paper, we will primarily consider unpunctured
admissible domains whose outer boundary component is the circle. We
will, without explicit mention, assume that any proper holomorphic map
between two such domains extends beyond the boundary and will denote the
extension by the same symbol.

\section{Bound on the size of the automorphism group}\label{s:aut}

We begin by showing that nontrivial automorphisms cannot have fixed
points on the boundary. We will give two proofs of this fact. The first
requires the following general result about the extension of
automorphisms that holds in all dimensions. We would like to
emphasise that the proof of this result is not particularly difficult, and the
most sophisticated tool used is Baire's category theorem.

\begin{result}[\cite{JK}]\label{R:Autext}
  Let $D \subset \C^n$ be a bounded domain. Assume that the automorphism
  group $\Aut(D)$ is compact and that each $f \in \Aut(D)$ extends
  holomorphically to some open set that contains $\ol{D}$. Then there
  exists an open neighbourhood $U$ of $\overline{D}$, which we can take
  to be a bounded domain, such that each $f \in \Aut(D)$ extends
  to $U$ as an automorphism.
\end{result}

In order to apply the above result, we need to know that the
automorphism group in our situation is compact. This is a version of the
famous Wong--Rosay theorem; see \cite[Theorem~12.2.3]{Krantz} for a
proof.

\begin{result}[Wong--Rosay in dimension $1$]
  Let $D$ be an unpunctured admissible domain of connectivity at least $2$.
  Then
  $\Aut(D)$ is compact.
\end{result}

We also require Cartan's uniqueness theorem; see
\cite[Proposition~12.1.1]{Krantz} for a proof.

\begin{result}\label{R:cartan}
  Let $D$ be a bounded domain and let $f:D \to D$ be holomorphic. If for
  some $a \in D$, we have $f(a) = a$, then $|f'(a)| \leq 1$. We have
  $|f'(a)| = 1$ if and only if $f$ is an automorphism. Furthermore, if
  $f'(a) = 1$, then $f \equiv \id$.
\end{result}

We now come to the key proposition that automorphisms other than the
identity cannot have fixed points on the boundary. The second proof,
which uses harmonic measure, is inspired by the arguments in a paper by
Heins \cite{Heins1}.

\begin{proposition}
\label{prop:no-boundary-fixed-point}
Let \(D\) be an unpunctured admissible domain of connectivity \(k\geq2\). If
\(f\in\Aut(D)\) fixes a point of \(\partial D\), then \(f=\id_D\).
\end{proposition}

\begin{proof}
Suppose that \(f(a)=a\), and let \(\Gamma_i\) be the boundary
component containing \(a\). It is clear that $f(\Gamma_i)=\Gamma_i$.
By Result~\ref{R:Autext}, $f$ extends to an automorphism of some
larger domain $\widetilde{D}$ that contains $\overline{D}$. By Cartan's
uniqueness theorem, we must have $|f'(a)| \leq 1$. Now $f(\Gamma_i)
= \Gamma_i$ and $f$ is orientation preserving as well. This forces
$f'(a) = 1$ and $f \equiv \id$ by the second part of Cartan's
uniqueness theorem.

We now come to the second proof. Let \(u\) be the harmonic measure of
\(\Gamma_i\). More precisely, \(u\) is the unique harmonic function on
\(D\) that is continuous on \(\overline D\) and satisfies
\[
    u=1\quad\text{on }\Gamma_i,
    \qquad
    u=0\quad\text{on }\partial D\setminus\Gamma_i.
\]
The functions \(u\circ f\) and \(u\) have the same boundary values.
Hence, uniqueness of the solution to the Dirichlet problem gives
$u\circ f \equiv u$.

We will now show that $u$ has regular values near $1$. First observe
that the reflection principle for harmonic functions guarantees that $u$
extends harmonically to a larger domain. Now consider the holomorphic
function $g:=u_x-i u_y$. Since \(u\) is nonconstant, \(g\not\equiv0\);
hence the zero set of $g$ is discrete and cannot accumulate at
$\Gamma_i$. We may thus choose $\eps > 0$ such that \(t:=1-\eps\) is a
regular value of \(u\). In fact, it is easy to see that we can choose
$\eps$ to ensure that the level set of \(u\) corresponding to the
regular value \(t\) is a real-analytic Jordan curve that separates
\(\Gamma_i\) from every other boundary component of \(D\).

Let $A$ denote the annular region in $D$ bounded by $\Gamma_i$ and the
level set. As $f$ preserves the level sets of $u$, we must have $f(A)
\subset A$. Similarly, $f^{-1}(A) \subset A$. Hence, $f(A) = A$ and
$f|_A$ is a biholomorphism of $A$. It is standard that a $2$-connected
domain is biholomorphic to an annulus. An elementary proof of this fact can be
found in \cite[Chapter~V]{Goluzin}. Let $\Phi:A \longrightarrow A_r$
be a biholomorphism onto the annulus with inner radius $r$, where
$0<r<1$, and outer radius $1$. Since
both boundary components of $A$ are real analytic, \(\Phi\) extends
biholomorphically across them. We now have the biholomorphism $\Phi
\circ f \circ \Phi^{-1}$ of $A_r$ that has a fixed point on the
boundary. This is impossible unless $f$ is the identity.
\end{proof}

\begin{remark}
  We have elected to give two proofs of the above proposition. The first
  proof requires the three results stated at the beginning of this
  section but the proofs of these results are not very hard
  and do not require deep machinery. Thus the first proof is in some
  sense more elementary. The second proof requires the machinery of
  harmonic measure which we will use later in the paper to obtain bounds
  on the number of proper holomorphic maps between two fixed finitely
  connected domains. So in a sense, the second proof is the more flexible one. 
\end{remark}

With boundary fixed points excluded, the rest of the proof is just
the argument principle: the number of fixed points of an automorphism is
given by the winding number of $f(\Gamma) - z$ about the origin, where
$\Gamma$ is the cycle obtained by summing up the finitely many Jordan
curves that comprise $\bdy D$ (taking into account the natural
orientation). Following He--Schramm \cite{Schramm}, we make the following

\begin{definition}
  Let \(\Gamma\) be an oriented Jordan curve and let \(h:\Gamma\to\C\) be
  continuous and fixed-point-free. The fixed-point index of $h$ is
  \[
  \Ind_\Gamma(h) := \wind_0(h(\Gamma(z))-\Gamma(z)).
  \]
  For a finite disjoint union of oriented curves, the index is defined
  as the sum of the indices over the curves.
\end{definition}

We will now count the number of fixed points that a map $f \in \Aut(D)$ can
possibly have by counting $\Ind_{\bdy D}(f)$, where $\bdy D$ is oriented
in the usual way. We first deal with those boundary curves that are fixed.
We need the following basic lemma.
\begin{lemma}\label{lem:inward-contraction}
Let \(\Gamma\) be a smooth real-analytic Jordan curve, and let \(J\) denote the
bounded component it determines. For every \(c\in J\), there is a continuous
map
\[
        \rho:\Gamma\times[0,1]\longrightarrow\ol J
\]
such that
\[
 \rho(z,0)=z,\qquad \rho(z,1)=c,
 \qquad \rho(z,t)\in J\quad(0<t\leq1).
\]
\end{lemma}

\begin{proof}
Let \(\phi:\D\to J\) be a Riemann map with \(\phi(0)=c\). By
Corollary~\ref{cor:biholo-extension}, \(\phi\) extends biholomorphically
across \(\T\), and \(\phi|_{\T}:\T\to\Gamma\) is a homeomorphism. For
\(z=\phi(\xi)\), \(\xi\in\T\), define
\[
        \rho(z,t)=\phi((1-t)\xi).
\]
This has all the stated properties.
\end{proof}

\begin{lemma}
Let \(\Gamma\) be a smooth real-analytic Jordan curve oriented
counterclockwise around the bounded component it determines, and let
\(h:\Gamma\to\Gamma\) be a fixed-point-free orientation-preserving
homeomorphism. Then
\[
        \Ind_\Gamma(h)=1.
\]
Thus, with the positive boundary orientation of a finitely connected domain, a
fixed outer component contributes \(+1\), whereas a fixed inner component
contributes \(-1\).
\end{lemma}

\begin{proof}
Choose \(c\) in the bounded component and let \(\rho\) be as in the
previous lemma. The homotopy of closed curves
\[
        h(\Gamma(z))-\rho(\Gamma(z),t)
\]
clearly misses $0$, and hence
\[
 \Ind_\Gamma(h) = \wind_0(h(\Gamma(z))-\Gamma(z))=\wind_0(h(\Gamma(z))-c)=1.
\]
The sign change for the inner components is due to their clockwise
orientation.
\end{proof}

For the next computation, let \(J_i\) be the bounded Jordan domain enclosed
by \(\Gamma_i\). Then \(D\subset J_0\), the domains
\(J_1,\ldots,J_{k-1}\) are pairwise disjoint, and
\(\ol J_i\subset J_0\) for \(i\geq1\).

\begin{lemma}\label{lem:component-table}
Let \(D\) be an unpunctured admissible domain of connectivity \(k\geq2\),
and let
\(h:\Gamma_i\to\Gamma_j\) be a homeomorphism which preserves the boundary
orientations induced by \(D\). Assume that \(h\) is fixed-point-free when
\(i=j\). Then
\[
\Ind_{\Gamma_i}(h)=
\begin{cases}
 +1,&i=j=0,\\
 -1,&i=j\neq0,\\
  0,&i\neq j\text{ and }i,j\neq0,\\
 +1,&i\neq j\text{ and }\{i,j\}\text{ contains }0.
\end{cases}
\]
\end{lemma}

\begin{proof}
The case $i = j$ follows from the previous lemma. Suppose that
$h:\Gamma_i \to \Gamma_j$, where $i \neq j$ and $i,j \neq 0$.
Contract \(\Gamma_i\) inside \(J_i\) to a point \(c_i\in J_i\) using
Lemma~\ref{lem:inward-contraction}. The final
curve is \(h(\Gamma_i(z)) - c_i\), and \(c_i\notin J_j\); it follows
that $\Ind_{\Gamma_i}(h) = 0$. If \(i = 0\) and \(j \neq 0\), we shrink
$h(\Gamma_i)$ to a point $c_j \in J_j \subset J_0$ to yield the curve
\(c_j-\Gamma_0(z)\). Thus $\Ind_{\Gamma_0}(h) = +1$ in this case.
Finally, suppose that \(i\neq0\) and \(j=0\). The curve
$h(\Gamma_i(z))$ has winding number $1$ with respect to any point in $J_0$.
The curve $\Gamma_i(z)$ can be shrunk to a point $c_i$ in $J_i \subset J_0$. Thus
$\Ind_{\Gamma_i}(h) = +1$ in this case as well.
\end{proof}

\begin{proposition}
\label{prop:cycle-index}
Let \(D\) be an unpunctured admissible domain of connectivity \(k\geq2\),
and let $f$ be a nontrivial automorphism. Denote by
\(b(f)\) the number of boundary components fixed setwise by \(f\), and
by \(N_D(f)\) the number of fixed points in \(D\), counted with
multiplicity. Then
\[
        N_D(f)=2-b(f).
\]
In particular, \(b(f)\leq2\) and \(N_D(f)\leq2\).
\end{proposition}

\begin{proof}
 Proposition
\ref{prop:no-boundary-fixed-point} shows that \(f\) has no fixed point on
\(\partial D\). The argument principle gives
\[
  N_D(f)=\Ind_{\partial D}(f).
\]

If the outer circle is fixed, then the previous lemma immediately gives
\[
  \Ind_{\partial D}(f) = 1 - (b(f) - 1) = 2 - b(f).
\]
Otherwise, the outer circle contributes $+2$ to the fixed-point index
(once as a source and once as a target) and each fixed inner component
contributes $-1$. The conclusion follows in this case as well.

\end{proof}

\begin{theorem}\label{thm:automorphism-bound}
Let \(D\) be an admissible domain of connectivity \(k\geq3\). Then
\[
        |\Aut(D)|\leq k(k-1)(k-2).
\]
Moreover, every nonidentity automorphism has at most two fixed points in
\(D\), counted with multiplicity.
\end{theorem}

\begin{proof}
 We can remove the punctures and assume that $D$ is unpunctured. In general, this
 reduces $k$. The cases when $D$ becomes a disc or a $2$-connected
 domain (in which case it is biholomorphic to an annulus) after filling
 the punctures are easy to handle. So we may assume $D$ has at least $3$
 boundary components even after filling in the punctures. Any
 nonidentity automorphism of $D$ can fix at most two boundary
 components. This immediately implies that the images of three distinct
 boundary components under an automorphism determine that automorphism
 uniquely. This yields the bound \(k(k-1)(k-2)\) for the size of the
 automorphism group. The fixed-point assertion is precisely the second
 conclusion of Proposition~\ref{prop:cycle-index}.
\end{proof}

\begin{remark}\label{R:old} In our paper \cite{bharathi}, we had shown
  the more general fact that any nontrivial automorphism of \textit{any}
  planar domain can have at most $2$ fixed points. We proved this by
  reducing to the case of finitely connected domains by showing that
  balls under the Kobayashi metric of a hyperbolic domain are finitely
  connected and then applying Koebe's circle mapping theorem. The above
  theorem gives us to bypass Koebe's theorem. We had also shown that the
  isotropy group of any bounded multiply connected domain is finite.
  This result also follows without the need to appeal to Koebe's
  theorem.
\end{remark}

\subsection{The proof of Heins's theorem (Result~\ref{R:heins})}

We will give the proof
assuming there are no punctures. The general case follows by filling in
the punctures and arguing as below.

Put \(H=\Aut(D)\) and \(N=|H|\). By the preceding theorem, \(H\) is
finite. We may assume that \(N>1\). Let \(\widehat D\) be the quotient of
\(\overline D\) obtained by
collapsing each \(\Gamma_j\) to a point \(p_j\). It is elementary that
\(\widehat D\) is homeomorphic to the sphere \(S^2\), but we do not need
this in the sequel.

Each \(h\in H\) extends to \(\overline D\), permutes the boundary
components, and therefore induces a homeomorphism \(\widehat h\) of
\(\widehat D\). The resulting action of \(H\) on \(\widehat D\) is
faithful. Moreover,
\[
 \operatorname{Fix}(\widehat h)
 =
 \operatorname{Fix}_D(h)\sqcup
 \{p_j:h(\Gamma_j)=\Gamma_j\}.
 \tag{1}
\]
Proposition~\ref{prop:cycle-index} and Cartan's uniqueness
theorem (Result~\ref{R:cartan}) therefore give
\[
    \#\operatorname{Fix}(\widehat h)=2,
    \qquad h\in H\setminus\{\mathrm{id}\}.
    \tag{2}
\]

Let
\[
    S:=\{x\in\widehat D:H_x\neq\{\mathrm{id}\}\},
\]
where $H_x$ is the stabiliser of $x$.
In other words, \(S\) is the set of points fixed by some nonidentity element of
\(H\). It is finite, since \(H\) is finite and, by~(2), every
nonidentity element fixes exactly two points. Let $S_1,\ldots,S_r$
be distinct \(H\)-orbits in \(S\), and let \(q_i\geq2\) be the order of the
stabiliser of a point in \(S_i\). By the orbit--stabiliser theorem,
$|S_i|=\frac{N}{q_i}$. We now perform the analogue of the classical pole count; compare \cite[\S6.12]{Artin}. Consider
\[
 \mathcal P
 :=
 \{(h,x)\in(H\setminus\{\mathrm{id}\})\times S:
                   \widehat h(x)=x\}.
\]
By~(2), every \(h\neq\mathrm{id}\) occurs in exactly two pairs, so
\[
    |\mathcal P|=2(N-1).
\]
On the other hand, every point of \(S_i\) is fixed by exactly
\(q_i-1\) nonidentity elements. Hence
\[
\begin{aligned}
    2(N-1)
      =\sum_{i=1}^r |S_i|(q_i-1) &=\sum_{i=1}^r\frac{N}{q_i}(q_i-1)\\
      &=N\sum_{i=1}^r\left(1-\frac1{q_i}\right).
\end{aligned}
\]
Consequently,
\[
    2-\frac2N
      =\sum_{i=1}^r\left(1-\frac1{q_i}\right).
    \tag{3}
\]

We now solve~(3). Since \(q_i\geq2\), every summand on the right is at
least \(1/2\). Since
\[
    1\leq2-\frac2N<2,
\]
it follows that \(r=2\) or \(r=3\).

If \(r=2\), then~(3) gives
\[
    \frac1{q_1}+\frac1{q_2}=\frac2N.
\]
Since \(q_i\leq N\), each term on the left is at least \(1/N\).
Equality therefore forces
\[
    q_1=q_2=N.
\]

Suppose \(r=3\), and arrange \(q_1\leq q_2\leq q_3\). Equation~(3)
becomes
\[
    \frac1{q_1}+\frac1{q_2}+\frac1{q_3}
       =1+\frac2N>1.
    \tag{4}
\]
If \(q_1\geq3\), the left-hand side is at most \(1\), so \(q_1=2\).
If \(q_2\geq4\), it is at most
\[
    \frac12+\frac14+\frac14=1,
\]
so \(q_2=2\) or \(3\).

If \(q_2=2\), equation~(4) gives
\[
    \frac1{q_3}=\frac2N.
\]
Thus
\[
    (q_1,q_2,q_3)=(2,2,\ell),
    \qquad N=2\ell.
\]

If \(q_2=3\), then~(4) implies \(q_3<6\). Hence
\(q_3=3,4,\) or \(5\), and substitution into~(4) gives respectively
\[
    N=12,\qquad N=24,\qquad N=60.
\]

We have therefore obtained the following exhaustive list:
\[
\begin{array}{c|c|c}
 (q_1,\ldots,q_r)
   & N
   & (|S_1|,\ldots,|S_r|)\\ \hline
 (N,N)
   & N
   & (1,1)\\
 (2,2,\ell)
   & 2\ell
   & (\ell,\ell,2)\\
 (2,3,3)
   & 12
   & (6,4,4)\\
 (2,3,4)
   & 24
   & (12,8,6)\\
 (2,3,5)
   & 60
   & (30,20,12).
\end{array}
\tag{5}
\]
Every orbit in
\(\widehat D\setminus S\) is free and therefore has size \(N\). The set
\[
    P:=\{p_0,\ldots,p_{k-1}\}
\]
of collapsed boundary components is \(H\)-invariant. Hence \(P\) is a
union of some of the nonfree orbits in~(5) and some free orbits.

In the first case of~(5), \(k\geq3\) forces \(P\) to contain a free
orbit, so \(N\leq k\). In the second case, \(P\) must contain either a
free orbit or an orbit of size \(\ell\). Therefore
\[
    N=2\ell\leq2k.
\]

It remains to consider the three exceptional cases. If \(N=12\) and
\(N>2k\), then \(k<6\); the nonfree orbit sizes \(6,4,4\) force
\(k=4\). If \(N=24\) and \(N>2k\), then \(k<12\); the orbit sizes
\(12,8,6\) force \(k=6\) or \(8\). Finally, if \(N=60\) and \(N>2k\),
then \(k<30\); the orbit sizes \(30,20,12\) force \(k=12\) or \(20\).
Thus
\[
 |\Aut(D)|\leq
 \begin{cases}
  12,&k=4,\\
  24,&k=6\ \text{or}\ 8,\\
  60,&k=12\ \text{or}\ 20,\\
  2k,&\text{otherwise}.
 \end{cases}
\]

The sharpness of the bound is established by the examples in
\cite{Heins2}.

\qed

\section{The Riemann--Hurwitz theorem for planar domains}

Let $G$ and $D$ be unpunctured admissible domains and let \(f:G\to D\) be a
proper holomorphic map of degree $d$. We have established in
Theorem~\ref{thm:proper-extension} that $f$ extends holomorphically to a
neighbourhood of $\overline{G}$ and that there are no critical points
on $\bdy G$. Thus the critical set of $f$ is finite and the number of
critical points counted with multiplicity is given by
\[
  r := \sum_{a\in G}\operatorname{ord}_a f'.
\]
We have also shown that the restriction of $f$ to each boundary component
of $G$ is a covering map.

\begin{lemma}\label{lem:boundary-degree} Let $f:G \to D$ be a proper
holomorphic map between unpunctured admissible domains.
Then, for every \(j\),
\[
        \sum_{\sigma(i)=j}d_i=d,
\]
where $d$ is the degree of $f$.
\end{lemma}

\begin{proof}
 For \(y\in\Gamma_j\), the set
$E=f^{-1}(y)\cap\ol G \subset \partial G$ comprises only regular points
by Theorem~\ref{thm:proper-extension}. Thus the
cardinality of this set is precisely the sum of the degrees $d_i$.
Let $x_1,\dots,x_\ell \in E$ and let $U_i$ be pairwise disjoint open sets
around $x_i$ in $\ol{G}$ on which $f$ is a diffeomorphism. Let $V_i :=
f(U_i)$. Set
\[
  V := \left(V_1 \cap \dots \cap V_\ell\right) \setminus f\left(\overline{G} \setminus \bigcup_i U_i\right).
\]
Then every point of $V\cap D$ is a regular value of $f$ and has exactly
$\ell$ preimages. This means $\ell = d$ and we are done.
\end{proof}

We now require Hopf's Umlaufsatz, which admits a very short proof in our
setting.

\begin{lemma}\label{lem:tangent-turning}
Let \(\Gamma\) be a smooth real-analytic Jordan curve that is oriented
positively with respect to the bounded domain it determines. Then
$\wind_0(\Gamma') = 1$. If $\Gamma$ were given the opposite orientation, then
$\wind_0(\Gamma') = -1$.
\end{lemma}

\begin{proof}

Let \(D\) be the bounded Jordan domain enclosed by \(\Gamma\), and choose a
Riemann map \(\phi:\D\to D\). By
Corollary~\ref{cor:biholo-extension}, \(\phi\) is biholomorphic on a
neighbourhood of \(\ol\D\). The curve
\(\gamma(t)=\phi(e^{it})\) is a reparametrisation of $\Gamma$, and
\[
  \gamma'(t)=ie^{it}\phi'(e^{it}).
\]
The function \(\phi'\) is holomorphic and nonvanishing on \(\ol\D\), so the
argument principle gives
\(\wind_0(\phi'(e^{it}))=0\). Hence \(\wind_0(\gamma')=1\). Reversing the
parameter reverses the winding number.
\end{proof}

\begin{theorem}
\label{thm:riemann-hurwitz}
Let \(G,D\subset\C\) be bounded finitely connected domains.
Suppose that \(G\) has \(m\)
boundary components, \(D\) has \(k\) boundary components, and
\(f:G\to D\) is a proper holomorphic map of degree \(d\). Then
\[
        m-2=d(k-2)+r,
\]
where $r$ is the number of critical points of $f$, counted with
multiplicity.
\end{theorem}

\begin{proof}
We will first assume that $G$ and $D$ are unpunctured admissible domains.
By Theorem~\ref{thm:proper-extension}, \(f\) is holomorphic near \(\ol G\)
and \(f'\neq0\) on \(\partial G\). The argument principle yields
\[
        r = \sum_{i = 0}^{m-1}\wind_0(f' \circ \Sigma_i).
\]
For a boundary component \(C\) of either $G$ or $D$, set
\[
 \varepsilon(C)=
 \begin{cases}
 +1,&C\text{ is the outer component},\\
 -1,&C\text{ is an inner component}.
 \end{cases}
\]
Hopf's Umlaufsatz gives
\[
        \wind_0(\Sigma_i')=\varepsilon(\Sigma_i).
\]
The curve \(f\circ\Sigma_i\) traverses
\(\Gamma_{\sigma(i)}\) exactly \(d_i\) times with the boundary orientation
of \(D\). Applying the Umlaufsatz again gives
\[
 \wind_0((f\circ \Sigma_i)')
 =d_i\varepsilon(\Gamma_{\sigma(i)}).
\]
Since
\((f\circ\Sigma_i)'=f'(\Sigma_i)\Sigma_i'\), additivity of winding under
multiplication gives
\[
 \wind_0(f'(z);\Sigma_i)
 =d_i\varepsilon(\Gamma_{\sigma(i)})-
  \varepsilon(\Sigma_i).
\]
Using Lemma~\ref{lem:boundary-degree}, we obtain
\[
\begin{aligned}
\sum_i d_i\varepsilon(\Gamma_{\sigma(i)})
 &=\sum_{j=0}^{k-1}\varepsilon(\Gamma_j)
      \sum_{\sigma(i)=j}d_i\\
 &=d\sum_{j=0}^{k-1}\varepsilon(\Gamma_j)
 =d(2-k),
\end{aligned}
\]
whereas \(\sum_i\varepsilon(\Sigma_i)=2-m\). Therefore
\[
  r=d(2-k)-(2-m),
\]
which is equivalent to the stated formula.

For the general case, let \(P_G\) and \(P_D\) denote the sets of
punctures of \(G\) and \(D\) with cardinalities $s$ and $t$,
respectively. Let $\widetilde f:\widetilde{G} \to \widetilde{D}$ denote the extended
proper holomorphic map between the domains obtained by filling in the punctures.
The connectivities of \(\widetilde G\) and \(\widetilde D\) are
\(m-s\) and \(k-t\), respectively. The degree of $\widetilde f$ remains
$d$. This means that the number of preimages of a point
in $P_D$ is $d$ (counted with multiplicity). These preimages can occur
only at the $s$ points of $P_G$. As there are $t$ points in $P_D$, it
follows that $\widetilde f$ has an additional $dt - s$ critical points.
Thus the number of critical points of \(\widetilde f\) is
\[
    \widetilde r := r+dt-s.
\]
The Riemann--Hurwitz formula for domains without punctures now gives
\[
\begin{aligned}
    m-s-2
      &=d(k-t-2)+r+dt-s \\
      &=d(k-2)+r-s.
\end{aligned}
\]
Adding \(s\) to both sides yields
\[
    m-2=d(k-2)+r,
\]
as required.
\end{proof}

\begin{corollary}
\label{cor:rh-automorphism-bound}
Let \(D\) be an unpunctured admissible domain of connectivity \(k\geq3\).
Then every proper holomorphic self-map of \(D\) is an
automorphism.
\end{corollary}

\begin{proof}
Let \(f:D\to D\) be a proper map of degree \(d\), and let \(r\) be the
number of its critical points, counted with multiplicity.
Theorem~\ref{thm:riemann-hurwitz} gives
\[
        k-2=d(k-2)+r.
\]
Since \(k-2>0\), \(d\geq1\), and \(r\geq0\), necessarily \(d=1\) and
\(r=0\). Every value is therefore regular and has exactly one preimage.
Thus \(f\) is bijective, and the inverse function theorem implies that it
is biholomorphic.
\end{proof}

\begin{corollary}\label{cor:degree-bound}
If \(k\geq3\), then every proper map \(f:G\to D\) satisfies
\[
        \deg f\leq\frac{m-2}{k-2}.
\]
In particular, no such map exists when \(m<k\).
\end{corollary}

\begin{proof}
This follows from Theorem~\ref{thm:riemann-hurwitz}, since \(r\geq0\) and
\(\deg f\geq1\).
\end{proof}

\section{Bounds on the size of the set of proper holomorphic maps}
\label{s:proper-bounds}

We require two standard results. The first is an elementary version of
the Hopf lemma. The second is the classical critical-point count for
harmonic measure due to Nevanlinna \cite[p. 36]{Nevanlinna}. We include short
proofs of both.

\begin{lemma}
\label{lem:hopf}
Let \(D\subset\mathbb C\) be a domain with \(C^{2}\)-smooth boundary,
let \(p\in\partial D\), and let \(\nu\) denote the inward unit normal at
\(p\). Suppose that
\[
    v\in C^{1}(\overline D)\cap C^{2}(D)
\]
is a nonconstant harmonic function satisfying
\[
    v\geq0\quad\text{in }D,
    \qquad
    v(p)=0.
\]
Then
\[
    \frac{\partial v}{\partial\nu}(p)>0.
\]
\end{lemma}

\begin{proof}
The strong minimum principle gives \(v>0\) in \(D\). Choose an interior
tangent disc such that
\[
    \overline{B(a,R)}\subset D\cup\{p\},
    \qquad
    \overline{B(a,R)}\cap\partial D=\{p\}.
\]
On the annulus
\[
    A=\{z:R/2<|z-a|<R\},
\]
consider the harmonic function
\[
    b(z)=\log\frac{R}{|z-a|}.
\]
Since \(v>0\) on the compact circle \(\{|z-a|=R/2\}\), there is an
\(\varepsilon>0\) such that \(v\geq\varepsilon b\) there. On
\(\{|z-a|=R\}\), we have \(b=0\leq v\). The maximum principle therefore
gives
\[
    v(z)\geq\varepsilon\log\frac{R}{|z-a|}
    \qquad (z\in A).
\]
Since \(a=p+R\nu\), it follows that
\[
\begin{split}
    \frac{\partial v}{\partial\nu}(p)
    &=
    \lim_{s\downarrow0}
    \frac{v(p+s\nu)-v(p)}{s} \\
    &\geq
    \varepsilon\lim_{s\downarrow0}
    \frac{1}{s}\log\frac{R}{R-s}
    =\frac{\varepsilon}{R}>0.
\end{split}
\]
\end{proof}

\begin{lemma}\label{lem:critical-points-harmonic-measure}
Let \(D\) be an unpunctured admissible domain of connectivity \(k\geq2\), with
\[
    \partial D=\Gamma_0\sqcup\cdots\sqcup\Gamma_{k-1},
\]
and let \(u_j\) be the harmonic measure of \(\Gamma_j\). Then \(u_j\)
has precisely \(k-2\) critical points in \(D\), counted with
multiplicity.
\end{lemma}

\begin{proof}
Let \(\tau\) be the positively oriented unit tangent, and let \(\nu=i\tau\)
be the inward unit normal. The critical points of $u_j$ are the zeros of
its complex derivative $h := (u_j)_x - i(u_j)_y$. At any boundary point
\(p\in\Gamma_i\), we have
\[
h =
\left(
\frac{\partial u_j}{\partial\tau}
-i\frac{\partial u_j}{\partial\nu}
\right)\overline{\tau}.
\]
The tangential derivative vanishes because \(u_j\) is constant on \(\Gamma_i\),
and by the Hopf lemma the normal derivative has a fixed nonzero sign
on each boundary component. The argument principle together with Hopf's
Umlaufsatz completes the proof.
\end{proof}

We now come to the proof of the main result of this paper.

\Proper*

\begin{proof}
Fix distinct boundary components \(\Gamma_1,\Gamma_2\) of \(D\) and
a boundary assignment \((A_1,A_2)\). Let \(\mathcal F\) be the
corresponding family of proper maps. We may assume that
\(\mathcal F\neq\varnothing\).

Let \(u_j\) be the harmonic measure of \(\Gamma_j\), and let \(v_j\)
be the harmonic function on \(G\) equal to \(1\) on \(A_j\) and to
\(0\) on the remaining boundary components. For every
\(f\in\mathcal F\), uniqueness of the solution to the Dirichlet problem gives
\begin{equation}
\label{eq:pullback}
u_j\circ f=v_j,\qquad j=1,2.
\end{equation}

Define
\[
P=\frac{(u_1)_z}{(u_2)_z},
\qquad
Q=\frac{(v_1)_z}{(v_2)_z}.
\]
These functions are meromorphic in neighbourhoods of
\(\overline D\) and \(\overline G\), respectively. Differentiating
\eqref{eq:pullback} gives, as an identity of meromorphic functions,
\begin{equation}
\label{eq:semiconjugacy}
P\circ f=Q,\qquad f\in\mathcal F.
\end{equation}

It follows from the arguments in the previous lemma that $P$ and $Q$ are
real-valued on the boundary of their respective domains. Moreover, neither
function is constant. Indeed, if \(P\equiv c\), then \(c\in\mathbb R\) and
\[
(u_1-cu_2)_z=0.
\]
Thus \(u_1-cu_2\) is constant; evaluating at a point of $\Gamma_1$ shows
that this constant is $1$. On any boundary component other than
\(\Gamma_1\) and \(\Gamma_2\), however, the left-hand side is $0$, a
contradiction. The same
argument applies to \(Q\) after observing that $A_1$ and $A_2$ together
cannot exhaust the boundary components of \(G\).

Since \(Q\) is nonconstant, \(Q(G)\) is open. Since \(P\) extends
meromorphically across \(\overline D\), it has only finitely many critical
values arising from a fixed neighbourhood of \(\overline D\). We may therefore
choose a nonreal number $c\in Q(G)$ that is a regular value of \(P\).
Now choose \(z_0\in G\) such that $Q(z_0)=c$.
For every \(f\in\mathcal F\), \eqref{eq:semiconjugacy} gives $P(f(z_0))=c$.
Every point of \(P^{-1}(c)\cap D\) is a zero of $(u_1)_z-c(u_2)_z$.
This function has exactly \(k-2\) zeros in \(D\), counted with
multiplicity. Indeed, it has no boundary zeros because \(c\notin
\mathbb R\) and an argument analogous to the one used in the proof of
Lemma~\ref{lem:critical-points-harmonic-measure} gives the claim. Consequently,
\[
\#\bigl(P^{-1}(c)\cap D\bigr)\leq k-2,
\]
and $f(z_0)$ must be one of these points. In particular, there are at
most \(k-2\) possibilities for \(f(z_0)\) as \(f\) varies over
\(\mathcal F\). We claim that the value of \(f(z_0)\) determines \(f\) uniquely.
Suppose that $g \in \mathcal F$ is another map such that
\[
f(z_0)=g(z_0)=w_0.
\]
Since \(c\) is a regular value of \(P\), the function \(P\) is
one-to-one on a neighbourhood of \(w_0\). From
\[
P\circ f=Q=P\circ g
\]
we obtain \(f=g\) near \(z_0\), and hence on \(G\). Thus
\[
\#\mathcal F\leq k-2.
\]

Each boundary component of \(G\) belongs to exactly one of three nonempty
classes: those mapped onto \(\Gamma_1\), those mapped onto
\(\Gamma_2\), and the remainder. The number of such assignments is
\[
3^m-3\cdot2^m+3.
\]
Therefore
\[
\boxed{
\#\{f:G\to D:f\text{ is proper holomorphic}\}
\leq
(k-2)\bigl(3^m-3\cdot2^m+3\bigr).
}
\]
If $G = D$ then any proper holomorphic self-map must necessarily have degree $1$
for otherwise we would get infinitely many proper holomorphic maps by iteration.
The stated bound on the size of the automorphism group is also immediate as
there are $k$ possible choices for the preimage of $\Gamma_1$ and $k-1$ for
$\Gamma_2$.

\end{proof}

\begin{remark}
  Our bound is quite loose but has the benefit of being clean. One can
  give tighter bounds by taking into account the number of boundary components
  of \(G\) that are mapped onto each boundary component of \(D\).
\end{remark}

\subsection*{Disclosure on the use of AI tools.}
ChatGPT-5.5 Plus, ChatGPT-5.6 Sol, and Claude Opus 5 were used in the
preparation of this article. We used these models for brainstorming, testing and
implementing ideas, writing preliminary versions of arguments, copyediting, and
proofreading. Ideas generated by the aforementioned models were used in the
proofs of Result~1.2, Theorem~1.3 and Theorem~4.3.
GitHub Copilot was used for LaTeX autocompletion. All AI generated arguments
have been independently verified by the author. The
author assumes full responsibility for the contents of this article. 

 \bibliographystyle{amsalpha}

\bibliography{proper}

\end{document}